\documentclass[12pt,reqno]{amsart}

\usepackage[centertags]{amsmath}
\usepackage{hyperref}
\usepackage{amsfonts}
\usepackage{amssymb}
\usepackage{amsthm}
\usepackage{newlfont}
\usepackage{amscd}
\usepackage{amsmath,amscd}

\usepackage[arrow,matrix]{xy}
\usepackage{verbatim}

\usepackage{color}

\newcommand{\QQ}{\mathbb{Q}}
\newcommand{\ZZ}{\mathbb{Z}}

\newcommand{\A}{\mathbb{A}}
\newcommand{\PP}{\mathbb{P}}

\newcommand{\GG}{\mathbb{G}}

\newcommand{\U}{\mcal{U}}

\def\U{{\mcal{U}}}
\def\O{{\mcal{O}}}

\def\vphi{{\varphi}}

\def\vphi{{\varphi}}

\def\OO{\mathcal{O}}

\newtheorem{thm}{Theorem}[section]

\newtheorem{cor}[thm]{Corollary}

\newtheorem{prop}[thm]{Proposition}

\theoremstyle{definition}
\newtheorem{define}[thm]{Definition}
\theoremstyle{remark}
\newtheorem{rem}[thm]{Remark}

\DeclareMathOperator{\Gal}{Gal}

\DeclareMathOperator{\Pic}{Pic}

\DeclareMathOperator{\Br}{Br}

\DeclareMathOperator{\Spec}{Spec}

\DeclareMathOperator{\et}{\acute{e}t}
\DeclareMathOperator{\sing}{sing}
\DeclareMathOperator{\crit}{crit}

\def\inv{\textup{inv}}

\def\T{\mathcal{T}}

\def\lrar{\longrightarrow}

\def\hrar{\hookrightarrow}

\def\val{\mathrm{val}}

\def \mcal{\mathcal}
\def \ovl{\overline}

\def \wtl{\widetilde}

\def \mcal {\mathcal}

\title[Singular curves and the \'etale Brauer--Manin obstruction]{Singular curves and the \'etale Brauer--Manin obstruction for surfaces}

\begin{document}
\date{\today}
\author{Yonatan Harpaz}
\address{Institute for Mathematics, Astrophysics and Particle Physics\\
Radboud University Nijmegen\\Heyendaalseweg 135\\
6525 AJ Nijmegen\\the Netherlands }
\email{harpazy@gmail.com}

\author{Alexei N. Skorobogatov}
\address{Department of Mathematics\\Imperial College London\\
South Kensington Campus\\SW7 2BZ\\U.K.}
\address{Institute for the Information Transmission Problems\\ Russian Academy of Sciences\\
19 Bolshoi Karetnyi\\ Moscow\\ 127994\\ Russia}
\email{a.skorobogatov@imperial.ac.uk}


\maketitle

\begin{abstract}
We give an elementary construction of a smooth and 
projective surface over an arbitrary number field $k$
that is a counterexample to the Hasse principle but has infinite
\'etale Brauer--Manin set. Our surface has a surjective morphism
to a curve with exactly one $k$-point
such that the unique $k$-fibre is
geometrically a union of projective lines
with an adelic point and the trivial Brauer group, but no $k$-point.

\medskip

\noindent Nous pr\'esentons une construction \'el\'ementaire d'une
surface lisse et projective sur un corps de nombres
quelconque $k$ qui constitue un contre-exemple au principe
de Hasse et poss\`ede l'ensemble de Brauer--Manin infini.
La surface est munie d'un morphisme surjectif vers une
courbe avec un seul $k$-point tel que l'unique fibre rationnelle,
qui g\'eom\'etriquement est l'union de droites projectives,
a un point ad\'elique et le groupe de Brauer
trivial, mais pas de $k$-points.
\end{abstract}

\section*{Introduction}

For a variety $X$ over a number field $k$ one can study the set
$X(k)$ of $k$-points of $X$ by embedding it into the topological space
of adelic points $X(\A_k)$. In 1970 Manin \cite{Manin} 
suggested to use the pairing
$$X(\A_k)\times\Br(X)\to\QQ/\ZZ$$ provided by local class field theory.
The left kernel of this pairing $X(\A_k)^{\Br}$ is a closed subset
of $X(\A_k)$, and the reciprocity law of global class field theory implies
that $X(k)$ is contained in $X(\A_k)^{\Br}$. The first example of
a smooth and projective variety $X$ such that $X(k)=\emptyset$
but $X(\A_k)^{\Br}\not=\emptyset$ was constructed in \cite{Sk1}
(see \cite{BSk} for a similar example; an 
earlier example conditional on the Bombieri--Lang
conjecture was found in \cite{SW}).
Later, Harari \cite{H} found many varieties $X$ such that $X(k)$ 
is not dense in $X(\A_k)^{\Br}$.
For all of these examples except for that of \cite{SW}
the failure of the Hasse principle
or weak approximation can be explained by the \'etale Brauer--Manin
obstruction (introduced in \cite{Sk1}, see also \cite{P1}):
the closure of $X(k)$ in $X(\A_k)$ is contained in the \'etale
Brauer--Manin set $X(\A_k)^{\et,\Br}\subset X(\A_k)^{\Br}$
which in these cases is smaller than $X(\A_k)^{\Br}$. 
Recently Poonen \cite{P1} 
constructed threefolds (fibred into rational surfaces over
a curve of genus at least 1) such that
$X(k)=\emptyset$ but $X(\A_k)^{\et,\Br}\not=\emptyset$. It is known
that $X(\A_k)^{\et,\Br}$ coincides with the set of adelic points
surviving the descent obstructions defined by torsors of arbitrary
linear algebraic groups (as proved in \cite{D, Sk2} using \cite{H2, Stoll}).

In 1997 Scharaschkin and the second author independently asked the question
whether $X(k)=\emptyset$ if and only if
$X(\A_k)^{\Br}=\emptyset$ when $X$ is a smooth and projective curve.
They also asked if the embedding of $X(k)$ into $X(\A_k)^{\Br}$ defines
a bijection between the closure of $X(k)$ in $X(\A_k)$ and 
the set of connected components of $X(\A_k)^{\Br}$.
Despite some evidence for these conjectures, it may be prudent
to consider also their weaker analogues with $X(\A_k)^{\et,\Br}$
in place of $X(\A_k)^{\Br}$.

In this note we give an elementary construction of a smooth and 
projective surface $X$ over an arbitrary number field $k$
that is a counterexample to the Hasse principle and has infinite
\'etale Brauer--Manin set (Section \ref{HP}).
Even simpler is our counterexample to weak approximation
(Section \ref{WA}).
This is a smooth and projective surface $X$ over $k$ with a
unique $k$-point and infinite \'etale Brauer--Manin set $X(\A_k)^{\et,\Br}$;
moreover, already the subset $Y(\A_k)\cap X(\A_k)^{\et,\Br}$ is infinite,
where $Y$ is the Zariski open subset of $X$ which is the complement to $X(k)$.
Following Poonen we consider families of curves parameterised
by a curve with exactly one $k$-point. The new idea is to make
the unique $k$-fibre a singular curve, geometrically 
a union of projective lines, and then use
properties of rational and adelic points on singular curves. 

The structure of the Picard group of a singular projective curve
is well known, see \cite[Section 9.2]{BLR} or \cite[Section 7.5]{Liu}. 
In Section \ref{B} we give a formula
for the Brauer group of a reduced projective curve, see 
Theorem \ref{p1}.
A singular curve over $k$ can have surprising properties that no
smooth curve can ever have: it can
contain infinitely many adelic points,
only finitely many $k$-points or none at all, and yet
have the trivial Brauer group.
See Corollary \ref{c1} for a singular,
geometrically connected, projective curve over 
an arbitrary number field $k$ that is a counterexample
to the Hasse principle not explained by the Brauer--Manin
obstruction. In \cite{Y} the first author
proves that every counterexample to the Hasse principle
on a curve which geometrically is a union of projective lines,
can be explained by finite descent (and hence by the \'etale
Brauer--Manin obstruction). 
Here we note that geometrically connected and simply connected projective
curves over number fields satisfy the Hasse principle,
a statement that does not generalise to higher dimension,
see Proposition \ref{p-basic} and Remark \ref{2.2}.

This paper was written while the authors were guests of
the Centre Interfacultaire Bernoulli of the Ecole Polytechnique
F\'ed\'erale de Lausanne.

\section{The Brauer group of singular curves} \label{B}

Let $k$ be a field of characteristic $0$ with an algebraic closure
$\bar k$ and the Galois group $\Gamma_k=\Gal(\bar k/k)$. 
For a scheme $X$ over $k$ we write $\ovl X=X\times_k\bar k$.
All cohomology groups in this paper are Galois or \'etale cohomology groups.
Let $C$ be a {\em reduced, geometrically connected, projective} curve over $k$. 
We define the {\em normalisation} $\wtl C$ as the disjoint union of 
normalisations of the irreducible components of $C$.  
The normalisation morphism $\nu:\wtl C\to C$ factors as
$$\wtl C\stackrel{\nu'}\lrar C'\stackrel{\nu''}\lrar C,$$ 
where $C'$ is a maximal intermediate curve
universally homeomorphic to $C$, see \cite[Section 9.2, p. 247]{BLR}
or \cite[Section 7.5, p. 308]{Liu}.
The curve $C'$ is obtained from $\wtl C$ by identifying the points
which have the same image in $C$.
In particular, there is a canonical
bijection $\nu'':C'(K)\tilde\lrar C(K)$ for any field extension $K/k$. 
The curve $C'$ has mildest possible singularities: for each singular
point $s\in C'(\bar k)$ the branches of $\ovl C'$ through $s$ intersect
like $n$ coordinate axes at $0\in \A^n_k$.

Let us define the following reduced $0$-dimensional schemes:
\begin{equation}
\Lambda=\Spec(H^0(\wtl C,\mathcal{O}_{\wtl C})), \ \ 
\Pi=C_{\mathrm{sing}}, \ \ 
\Psi=\big(\Pi\times_C\wtl C\big)_{\mathrm{red}}.
\label{ee3}
\end{equation}
Here $\Lambda$ is the $k$-scheme of geometric irreducible components of $C$
(or the geometric connected components of $\wtl C$); it is
the disjoint union of closed points $\lambda=\Spec(k(\lambda))$
such that $k(\lambda)$ is the algebraic closure of $k$ 
in the function field 
of the corresponding irreducible component $k(C_\lambda)=k(\wtl C_\lambda)$.
Next, $\Pi$ is the union of singular points of $C$, and 
$\Psi$ is the union of fibres of $\nu: \wtl C\to C$ over the singular
points of $C$ with their reduced subscheme structure. 
The morphism $\nu''$ induces an isomorphism 
$\big(\Pi\times_C C'\big)_{\mathrm{red}}\tilde\lrar \Pi$, so we can
identify these schemes.
Let $i:\Pi\to C$, $i':\Pi\to C$ and $j:\Psi\to \wtl C$ be the natural
closed immersions. We have a commutative diagram
$$ \xymatrix{
\wtl C \ar^{\nu'}[r]& C' \ar^{\nu''}[r]& C\\
\Psi \ar^{j}[u]\ar^{\nu'}[r] & \Pi \ar^{i'}[u]\ar_{i}[ur]&\\
}$$
The restriction of $\nu$ to the smooth locus
of $C$ induces isomorphisms 
$$\wtl C\setminus j(\Psi)\tilde\lrar C'\setminus i'(\Pi) 
\tilde\lrar C\setminus i(\Pi).$$

An algebraic group over $\Pi$ is a product
$G=\prod_\pi i_{\pi*}(G_\pi)$, where $\pi$ ranges over the irreducible 
components of $\Pi$, $i_\pi: \Spec(k(\pi))\to\Pi$ is the natural
closed immersion, and $G_\pi$ is an algebraic group over the field $k(\pi)$.

\begin{prop}\label{exact}
$(i)$ The canonical maps $\GG_{m,C'}\to \nu'_*\GG_{m,\wtl C}$ and
$\GG_{m,C'}\to i'_*\GG_{m,\Pi}$ give rise
to the exact sequence of \'etale sheaves on $C'$
\begin{equation}
0\to\GG_{m,C'}\to \nu'_*\GG_{m,\wtl C}\oplus i'_*\GG_{m,\Pi}\to i'_*\nu'_*\GG_{m,\Psi}
\to 0, \label{ee1}
\end{equation}
where $\nu'_*\GG_{m,\Psi}$ is an algebraic torus over $\Pi$.

$(ii)$ The canonical map $\GG_{m,C}\to \nu''_*\GG_{m,C'}$ gives rise
to the exact sequence of \'etale sheaves on $C$:
\begin{equation}
0\to\GG_{m,C}\to \nu''_*\GG_{m,C'}\to i_*\U \to 0, \label{ee2}
\end{equation}
where $\U$ is a commutative unipotent group over $\Pi$.
\end{prop}
\begin{proof}
This is essentially well known, see \cite{BLR}, the proofs of 
Propositions 9.2.9 and 9.2.10,
or \cite[Lemma 7.5.12]{Liu}. By \cite[Thm. II.2.15 (b), (c)]{EC} 
it is enough to check the exactness of (\ref{ee1}) 
at each geometric point $\bar x$ of $C'$. 
If $\bar x\notin i'(\Pi)$, this is obvious since
locally at $\bar x$ the morphism $\nu'$ is an isomorphism, and the 
stalks $(i'_*\GG_{m,\Pi})_{\bar x}$ and 
$(i'_*\nu'_*\GG_{m,\Psi})_{\bar x}$ are zero.
Now let $\bar x\in i'(\Pi)$, and let $\O_{\bar x}$ be the strict
henselisation of the local ring of $\bar x$ in $C'$. 
Each geometric point $\bar y$ of $\wtl C$ belongs to exactly one
geometric connected component of $\wtl C$, and we denote by
$\O_{\bar y}$ the strict henselisation of the local ring of $\bar y$ in its
geometric connected component. By the construction of $C'$ 
we have an exact sequence 
$$0\lrar \O_{\bar x}\lrar k(\bar x)\times
\prod_{\nu'(\bar y)=\bar x}\O_{\bar y}\lrar 
\prod_{\nu'(\bar y)=\bar x} k(\bar y)\lrar 0,$$
where $\O_{\bar y}\to k(\bar y)$ is the reduction modulo
the maximal ideal of $\O_{\bar y}$, and $k(\bar x)\to k(\bar y)$
is the multiplication by $-1$. 
We obtain an exact sequence of abelian groups
$$1\lrar \O_{\bar x}^*\lrar k(\bar x)^*\times
\prod_{\nu'(\bar y)=\bar x}\O_{\bar y}^*\lrar 
\prod_{\nu'(\bar y)=\bar x} k(\bar y)^*\lrar 1.$$
Using \cite[Cor. II.3.5 (a), (c)]{EC} one sees that
this is the sequence of stalks of (\ref{ee1}) at
$\bar x$, so that $(i)$ is proved.

To prove $(ii)$ consider the exact sequence 
$$0\to\GG_{m,C}\to \nu''_*\GG_{m,C'}\to \nu''_*\GG_{m,C'}/\GG_{m,C}
\to 0.$$
Since $\nu''$ is an isomorphism away from $i(\Pi)$, the restriction
of the sheaf $\nu''_*\GG_{m,C'}/\GG_{m,C}$ to $C\setminus i(\Pi)$
is zero, hence $\nu''_*\GG_{m,C'}/\GG_{m,C}=i_*\U$
for some sheaf $\U$ on $\Pi$. To see that $\U$ is a unipotent
group scheme it is enough to check the stalks at geometric points.
Let $\bar x$ be a geometric point of $i(\Pi)$, and let
$\bar y$ be the unique geometric point of $C'$ such that
$\nu''(\bar y)=\bar x$. Let $\O_{\bar x}$ and $\O_{\bar y}$
be the corresponding strictly henselian local rings.
The stalk $(\nu''_*\GG_{m,C'}/\GG_{m,C})_{\bar x}$ is 
$\O_{\bar y}^*/\O_{\bar x}^*$, and according to \cite[Lemma 7.5.12 (c)]{Liu},
this is a unipotent group over the field $k(\bar x)$.
This finishes the proof.
\end{proof}

\begin{rem}
The first part of Proposition~\ref{exact} has an alternative proof which is easier to generalise to higher dimension. Let $X$ be a projective $k$-variety with 
normalisation morphism $\nu: \wtl{X} \to X$. Assume that $\wtl{X}$, $X_{\sing}$ and $\wtl{X}_{\crit}$ are smooth, where $X_{\sing}$ is the singular locus of $X$ and $\wtl{X}_{\crit} = \nu^{-1}(X_{\sing}) \subseteq \wtl{X}$ is the critical locus of $\nu$. (This assumption is automatically satisfied when $X$ is a curve.) The analogue of $C'$ is the $K$-variety $X'$ given by the pushout in the square
$$ \xymatrix{
\wtl{X}_{\crit} \ar^{j}[r]\ar_{g}[d] & \wtl{X} \ar^{\nu'}[d] \\
X_{\sing} \ar^{i'}[r] & X' \\
}$$
This pushout exists in the category of $K$-varieties since $i'$ is a closed embedding and $g$ is an affine morphism of smooth projective varieties 
(see~\cite[Thm. $5.4$]{FE}). One then proves that the sequence of sheaves
$$ 0 \lrar \GG_{m,X'} \lrar \nu'_*\GG_{m,\wtl{X}} \oplus i'_*\GG_{m,X_{\sing}} \lrar \nu'_*j_*\GG_{m,\wtl{X}_{\crit}} \lrar 0 $$
is exact, as follows. From the definition of $X'$ we obtain that the square
$$ \xymatrix{
\OO_{X'} \ar[r]\ar[d] & \nu'_*\OO_{\wtl{X}} \ar[d] \\
i'_*\OO_{X_{\sing}} \ar[r] & \nu'_*j_*\OO_{\wtl{X}_{\crit}} \\
}$$
is a Cartesian diagram of sheaves of rings on $X'$ (where Cartesian for sheaves means Cartesian on the stalks). The functor $\GG_m$ from the category of rings 
with $1$ to the 
category of abelian groups that associates to a ring its group of units, 
commutes with limits and filtered colimits (e.g. with taking stalks). 
This implies that the diagram
$$ \xymatrix{
\GG_{m,X'} \ar[r]\ar[d] & \nu'_*\GG_{m,\wtl{X}} \ar[d] \\
i'_*\GG_{m,X_{\sing}} \ar[r] & \nu'_*j_*\GG_{m,\wtl{X}_{\crit}} \\
}$$
is Cartesian. Hence the exactness of the sequence
$$ 0 \lrar \GG_{m,X'} \lrar \nu'_*\GG_{m,\wtl{X}} \oplus i'_*\GG_{m,X_{\sing}} \lrar \nu'_*j_*\GG_{m,\wtl{X}_{\crit}}. $$
It remains to check that the last map here is surjective. The morphism 
$\nu$ is finite, hence the functor $\nu_*$ is exact \cite[Cor. II.3.6]{EC}.
The map
$\GG_{m,\wtl{X}} \to j_*\GG_{m,\wtl{X}_{\crit}}$ is surjective,
because $j:\wtl{X}_{\crit}\to \wtl{X}$ is a closed embedding, and thus
$\nu'_*\GG_{m,\wtl{X}} \to \nu'_*j_*\GG_{m,\wtl{X}_{\crit}}$ is 
surjective too.
\end{rem}

For fields $k_1,\ldots,k_n$, we have $\Br\big(\coprod_{i=1}^n\Spec(k_i)\big)=\oplus_{i=1}^n\Br(k_i)$.

\begin{thm} \label{p1}
Let $C$ be a reduced, geometrically connected, projective curve,
and let $\Lambda$, $\Pi$ and $\Psi$ be the schemes defined
in $(\ref{ee3})$. Let $\Lambda=\coprod_\lambda \Spec(k(\lambda))$
be the decomposition into the disjoint union of 
connected components, so that $\wtl C=\coprod_\lambda \wtl C_\lambda$,
where $\wtl C_\lambda$ is a geometrically integral, smooth, projective
curve over the field $k(\lambda)$. Then there is an exact sequence
\begin{equation}
0\lrar \Br(C)\lrar \Br(\Pi)\oplus\bigoplus_{\lambda\in\Lambda}
\Br(\wtl C_\lambda) \lrar \Br(\Psi),
\label{ee4}
\end{equation}
where the maps are the composition of canonical maps
$$\Br(\wtl C_\lambda)\to \Br(\wtl C_\lambda\cap \Psi)
\to \Br(\Psi),$$
and the opposite of the restriction map $\Br(\Pi)\to \Br(\Psi)$.
\end{thm}
\begin{proof}
Let $\pi$ range over the irreducible components of $\Pi$, so that
$\U=\prod_{\pi} i_{\pi*}(U_\pi)$, where $U_\pi$ is a 
commutative unipotent group over the field $k(\pi)$. 
Since $i_*$ is an exact functor \cite[Cor. II.3.6]{EC}, 
we have $H^n(C,i_*\U)=H^n(\Pi,\U)=\prod_\pi H^n(k(\pi),U_\pi)$.
The field $k$ has characteristic 0, and it is well known that this implies
that any commutative unipotent group has zero cohomology in degree
$n>0$. (Such a group has a composition
series with factors $\GG_a$, and $H^n(k,\GG_a)=0$ for any $n>0$,
see \cite[X, Prop. 1]{CL}.)
Thus the long exact sequence of cohomology groups associated to (\ref{ee2})
gives rise to an isomorphism $\Br(C)=H^2(C,\GG_{m,C}) \tilde\lrar
H^2(C,\nu''_*\GG_{m,C'})$. Since $\nu''$ is finite, the functor 
$\nu''_*$ is exact 
\cite[Cor. II.3.6]{EC}, so we obtain an isomorphism
$\Br(C)\tilde\lrar \Br(C')$. We now apply similar arguments to
(\ref{ee1}). Hilbert's theorem 90 gives 
$H^1(\Pi,\nu'_*\GG_{m,\Psi})=H^1(\Psi,\GG_{m,\Psi})=0$,
so we obtain the exact sequence (\ref{ee4}).
\end{proof}
Recall that a reduced, geometrically connected, projective curve $S$ 
over a field $k$ is called {\em semi-stable}
if all the singular points of $S$ are ordinary double points \cite[Def. 9.2.6]{BLR}.

\begin{define} \label{bi}
A semi-stable curve is called {\em bipartite} if it is a union of 
two smooth curves without common irreducible components.
\end{define}

\begin{cor} \label{p3}
Let $S=S^+\cup S^-$ be a bipartite curve, where $S^+$ and $S^-$ are smooth
curves such that $S^+\cap S^-$ is finite.
Then there is an exact sequence
\begin{equation}
0\lrar\Br(S)\lrar\Br(S^+)\oplus\Br(S^-)\lrar \Br(S^+\cap S^-),
\label{ee5}
\end{equation}
where $\Br(S)\to\Br(S^+)\oplus\Br(S^-)$ is the natural map, and
$\Br(S^\pm)\to\Br(S^+\cap S^-)$ is the restriction
map multiplied by $\pm 1$.
\end{cor}
\begin{proof}
In our previous notation we have $\wtl S=S^+\coprod S^-$,
$\Pi=S^+\cap S^-$,  and $\Psi$ is the disjoint union of two
copies of $\Pi$, namely $\Psi^+=\Psi\cap S^+$ and $\Psi^-=\Psi\cap S^-$.
In particular, the restriction map $\Br(\Pi)\to\Br(\Psi)$ is injective.
Thus taking the quotients by $\Br(\Pi)$ in the middle and last terms of
(\ref{ee4}), we obtain (\ref{ee5}).
\end{proof}

The constructions in Sections \ref{WA} and \ref{HP} use singular curves 
of the following special kind.

\begin{define}
A reduced, geometrically connected, projective curve $C$ over $k$ is called {\em conical}
if every irreducible component of $\ovl C$ is rational.
\end{define}

For {\em bipartite conical} curves the calculation of the Brauer group
can be carried out using only the Brauer groups of fields.

\begin{cor} \label{c4}
Let $C=C^+\cup C^-$ be a bipartite conical curve, and let 
$\Lambda=\Lambda^+\coprod\Lambda^-$ be the corresponding 
decomposition of $\Lambda$.
Then $\Br(C)$ is the kernel of the map
$$\bigoplus_{\lambda\in\Lambda^+}\Br(k(\lambda))/[\wtl C_\lambda]\ \oplus\
\bigoplus_{\lambda\in\Lambda^-}\Br(k(\lambda))/[\wtl C_\lambda]\
\lrar \Br(C^+\cap C^-),$$
where $[\wtl C_\lambda]\in \Br(k(\lambda))$
is the class of the conic $\wtl C_\lambda$ over the field $k(\lambda)$,
and the map 
$\Br(k(\lambda))/[\wtl C_\lambda]\to \Br(C^+\cap C^-)$ is
the restriction followed by multiplication by $\pm 1$ when $\lambda\in\Lambda^\pm$.
\end{cor}
\begin{proof}
This follows directly 
from Proposition \ref{p3} and the well known fact that 
the Brauer group of a conic over $k$ is the quotient of $\Br(k)$
by the cyclic subgroup generated by the class of this conic.
\end{proof}

In some cases one can compute $\Br(C)$ using the Hochschild--Serre
spectral sequence $H^p(k,H^q(\ovl C,\GG_m))\Rightarrow H^{p+q}(C,\GG_m)$.
In~\cite[III, Cor. 1.2]{GB} Grothendieck proved that $\Br(\ovl{C}) = 0$
for any curve $C$. Hence the spectral sequence identifies
the cokernel of the natural map
$\Br(k)\to \Br(C)$ with a subgroup of $H^1(k,\Pic(\ovl C))$.

The structure of $\Gamma_k$-module $\Pic(\ovl C)$ is well known,
at least up to its maximal unipotent subgroup. It is convenient
to describe this structure in combinatorial terms.
Recall that $\ovl\Lambda$, $\ovl \Pi$, $\ovl \Psi$ are the $\bar k$-schemes
obtained from $\Lambda$, $\Pi$, $\Psi$ by extending the ground field 
to $\bar k$.
We associate to $C$ the {\em incidence graph} $X(C)$ defined as the
directed graph whose vertices are $X(C)_0=\ovl \Lambda\cup\ovl \Pi$
and the edges are $X(C)_1=\ovl \Psi$. The edge $Q \in \ovl \Psi$
goes from $L\in \ovl\Lambda$ to $P \in \ovl \Pi$ when $\nu(Q)=P$ and
$Q$ is contained in the irreducible component
$L$ of $\wtl C$. The source and target maps 
$X(C)_1 \to X(C)_0$ can be described as a morphism of $k$-schemes
$$(s,t):\Psi \lrar \Lambda \coprod \Pi,$$
where $t:\Psi\to\Pi$ is induced by $\nu'$, and $s$ is the
composition of the closed immersion $j:\Psi\to\wtl C$ and the canonical morphism
$\wtl C\to\Lambda$.
By construction $X(C)$ is a connected bipartite graph with a
natural action of the Galois group $\Gamma_k$.

For a reduced 0-dimensional $k$-scheme $p:\Sigma\to\Spec(k)$
of finite type, the $k$-group scheme $p_*\GG_{m}$
is an algebraic torus over $k$. If we write
$\Sigma=\coprod_{i=1}^n\Spec(k_i)$, where $k_1,\ldots,k_n$ are finite
field extensions of $k$, then $p_*\GG_{m}$ is
the product of Weil restrictions $\prod_{i=1}^n R_{k_i/k}(\GG_m)$. 
For a reduced 0-dimensional scheme $p':\Sigma'\to\Spec(k)$
of finite type a morphism of $k$-schemes
$f:\Sigma'\to \Sigma$ gives rise to a canonical morphism
$\GG_{m,\Sigma}\to f_*\GG_{m,\Sigma'}$, and hence to
a canonical homomorphism of $k$-tori $p_*\GG_{m}\to p'_*\GG_{m}$,
which we denote by $f^*$.
Let us denote the structure morphism $\Lambda\to\Spec(k)$ by $p_\Lambda$,
and use the same convention for $\Psi$ and $\Pi$.
Let $T$ be the algebraic $k$-torus defined by the exact sequence
$$
1\lrar \GG_{m}\lrar p_{\Lambda *}\GG_m\oplus p_{\Pi *}\GG_m\lrar
p_{\Psi *}\GG_m\lrar T\lrar 1,
$$
where the middle arrow is $s^*(t^*)^{-1}$.
\medskip

\begin{rem} The character group $\hat T$ with its
natural action of the Galois group
$\Gamma_k$, is canonically isomorphic to 
$H_1(X(C),\ZZ)$, the {\em first homology group} of the graph
$X(C)$. Since $X(C)$ is connected,
we have $T=\{1\}$ if and only if $X(C)$ is a tree.
\end{rem}

\begin{prop} \label{P1}
Let $C$ be a reduced, geometrically connected, projective curve over $k$.
We have the following exact sequences of $\Gamma_k$-modules:
\begin{equation}
0\lrar T(\bar k)\lrar \Pic(\ovl {C'})\lrar \Pic(\wtl C\times_k\bar k)
\lrar 0,
\label{b}
\end{equation}
\begin{equation}
0\lrar U(\bar k)\lrar \Pic(\ovl C)\lrar \Pic(\ovl {C'})\lrar 0,
\label{a}
\end{equation}
where $U$ is a commutative unipotent algebraic group over $k$.
\end{prop}
\begin{proof} 
To obtain (\ref{b}) we apply the direct image functor
with respect to the structure morphism $C'\to\Spec(k)$
to the exact sequence (\ref{ee1}).
The sequence (\ref{a}) is obtained from (\ref{ee2}) in a similar way. 
\end{proof}

\begin{cor} \label{C1}
If $C$ is a conical curve such that $X(C)$ is a tree, then 
$H^1(k,\Pic(\ovl C))=0$, so that the natural map
$\Br(k)\to\Br(C)$ is surjective.
\end{cor}
\begin{proof}
Since $k$ has characteristic 0, we have $H^n(k,U)=0$ for $n>0$. 
Thus (\ref{a}) gives an isomorphism 
$H^1(k,\Pic(\ovl C))=H^1(k,\Pic(\ovl {C'}))$.
We have $T=\{1\}$, hence $\Pic(\ovl {C'})= \Pic(\wtl C\times_k\bar k)$.
The curve $\wtl C$ is the disjoint union of conics defined
over finite extensions of $k$, thus the free abelian group
$\Pic(\wtl C\times_k\bar k)$ has a natural $\Gamma_k$-stable $\ZZ$-basis.
Hence $H^1(k,\Pic(\wtl C\times_k\bar k))=0$.
\end{proof}

\section{Weak approximation} \label{WA}

Let $k$ be a number field. Recall that the \'etale Brauer--Manin
set $X(\A_k)^{\et,\Br}$ is the set of adelic points $(M_v)\in X(\A_k)$
satisfying the following property:
for any torsor $f:Y\to X$ of a finite $k$-group scheme $G$
there exists a $k$-torsor $Z$ of $G$ such that $(M_v)$
is the image of an adelic point in the Brauer--Manin set of $(Y\times_k Z)/G$.
Here $G$ acts simultaneously on both factors, and the morphism
$(Y\times_k Z)/G\to X$ is induced by $Y\to X$. It is clear that
the \'etale Brauer--Manin obstruction is a functor from the category of 
varieties over $k$ to the category of sets. Note that
$(Y\times_k Z)/G\to X$ is a torsor of an inner form of $G$, called
the twist of $Y/X$ by $Z$, see \cite[pp. 20--21]{Sk} for details.

In this section we construct a simple example
of a smooth projective surface $X$ over $k$ such that 
$X(\A_k)^{\et,\Br}$ is infinite but $X(k)$ contains only one point. Thus
$X(k)$ is far from dense in $X(\A_k)^{\et,\Br}$; in fact, 
infinitely many points 
of $X(\A_k)^{\et,\Br}$ are contained in the complement to $X(k)$ in $X$.

We start with the following statement which shows that on an everywhere locally
soluble conical curve $C$ such that $X(C)$ is a tree
all the adelic points survive the \'etale Brauer--Manin obstruction.

\begin{prop}\label{p-basic}
Let $k$ be a number field, and let $C$ be a conical curve
over $k$ such that $X(C)$ is a tree and $C(\A_k)\not=\emptyset$.
Then
\begin{enumerate}
\item $C(k)\not=\emptyset$;
\item the natural map $\Br(k)\to\Br(C)$ is an isomorphism;
\item $C(\A_k)^{\et,\Br} = C(\A_k)$.
\end{enumerate}
\end{prop}
\begin{rem} \label{2.2}
Proposition \ref{p-basic} (1) implies that
{\em geometrically connected and simply connected projective
curves over number fields satisfy the Hasse principle}.
It is easy to see that this statement 
does not generalise to higher dimension. Consider
a conic $C\subset\PP^2_k$ without a $k$-point, and choose
a quadratic extension $K/k$ so that all the places $v$ 
with $C(k_v)=\emptyset$ are split in $K$. Then
the union of two planes conjugate over $K$ and intersecting
at $C$ is a geometrically connected and 
simply connected projective surface that is a counterexample
to the Hasse principle.
\end{rem}

\medskip

\begin{proof}[Proof of Proposition \ref{p-basic}]
Let us prove (1).
It is well known that any group acting on a finite connected tree
fixes a vertex or an edge. (This is easily proved by induction
on the diameter of a tree, that is, on the length of a longest
path contained in it.)
We apply this to the action of the Galois
group $\Gamma_k$ on $X(C)$. If $\Gamma_k$ fixes a point of $\ovl\Pi$
or $\ovl\Psi$, then $C(k)\not=\emptyset$.
If $\Gamma_k$ fixes a point of $\ovl\Lambda$, then $C$ has an
irreducible component $C_0$ which is a geometrically integral
geometrically rational curve. Let $\wtl C_0$ be the normalisation of $C_0$.
Since $X(C)$ is a tree, the morphism $\wtl C_0\to C_0$ 
is a bijection on points.
Thus if we can prove that $C_0(\A_k)\not=\emptyset$, then 
$\wtl C_0(\A_k)\not=\emptyset$, and by the Hasse--Minkowski theorem
$\wtl C_0(k)\not=\emptyset$, so that finally $C_0(k)\not=\emptyset$.
Since $X(C)$ is a connected tree, each connected component of 
$\ovl C\setminus \ovl C_0$ meets $\ovl C_0$ in exactly one point. 
Let $k_v$ be a completion of $k$.
Suppose that $C_0(k_v)=\emptyset$. Since $C(k_v)\not=\emptyset$,
at least one of the connected components of 
$\ovl C\setminus \ovl C_0$ is fixed by the Galois group $\Gamma_{k_v}$,
and hence it intersects $C_0$ in a $k_v$-point. This contradiction
proves (1).

By (1) the natural map $\Br(k)\to\Br(C)$ has a retraction, and so is injective,
but it is also surjective by Corollary \ref{C1}. This proves (2).

Let us prove (3).
Let $G$ be a finite $k$-group scheme, and let $\T \to C$ be
a torsor of $G$. Fix a $k$-point $P$ in $C$.
The fibre $\T_P$ is then a $k$-torsor.
The twist of $\T$ by $\T_P$ is
the quotient of $\T\times_k \T_P$ by the diagonal action of $G$.
This is a $C$-torsor of an inner form of $G$ such that
the fibre at $P$ has a $k$-point, namely the quotient by $G$ of the diagonal
in $\T_P\times_k \T_P$.
Thus, twisting $\T$ by a $k$-torsor of $G$, and replacing $G$ by an
inner form we can assume that $\T_P$ contains a $k$-point $Q$.
Since all irreducible components of $\ovl C$ are homeomorphic
to $\PP^1_{\bar k}$, the torsor $\ovl\T\to \ovl C$
trivialises over each component of $\ovl C$. But $X(C)$ is a tree, 
and this implies that the torsor
$\ovl\T\to \ovl C$ is trivial, that is,
$\ovl \T\simeq (C\times_k G)\times_k\bar k$. The connected
component of $\ovl\T$ that contains $Q$ is thus defined
over $k$, and hence it gives a section $s$ of $\T\to C$ such that 
$s(P)=Q$. Hence any adelic point on $C$ lifts to an
adelic point on $s(C)\subset \T$. Since $\Br(C)=\Br(k)$
we conclude that $C(\A_k)$ is contained in, and hence is equal to the \'etale
Brauer--Manin set $C(\A_k)^{\et,\Br}$.
\end{proof}

Let $k$ be a number field, and let $f(x,y)$ be a separable homogeneous
polynomial such that its
zero locus $Z^f \subseteq \PP^1_{k}$ is a $0$-dimensional scheme
violating the Hasse principle. It is easy to see that such a polynomial
exists for any number field $k$. For example, for $k=\QQ$ one can take
\begin{equation}
 f(x,y) = (x^2-2y^2)(x^2-17y^2)(x^2-34y^2).\label{e1}
\end{equation}
For an arbitrary number field $k$ the following polynomial will do:
$$f(x,y) = (x^2-ay^2)(x^2-by^2)(x^2-aby^2)(x^2-cy^2),$$
where $a,b,c\in k^*\setminus k^{*2}$ are such that $ab\notin k^{*2}$,
whereas $c\in k_v^{*2}$
for all places $v$ such that $\val_v(a)\not=0$ or $\val_v(b)\not=0$, 
and also for
the archimedean places, and the places with residual characteristic 2. (For
fixed $a$ and $b$ one finds $c$ using weak approximation.) Let $d=\deg(f)$.

Let $C^f \subseteq \PP^2_k$ be the curve with the equation
$f(x,y) = 0$. It is geometrically connected and has a unique singular point
$P = (0:0:1) \in C^f \subset \PP^2_k$ which is contained in
all irreducible components of $C^f$.
Since $Z^f(k)=\emptyset$ we see that $P$ is the only
$k$-point of $C^f$. The intersection of $C^f$ with any line
in $\PP^2_k$ that does not pass through $P$ is isomorphic to $Z^f$,
hence {\em the smooth locus of $C^f$ contains an infinite subset
of $C^f(\A_k)=C^f(\A_k)^{\et,\Br}$}, where the equality is due to 
Proposition~\ref{p-basic} (3).

Now let $g(x,y,z)$ be a homogeneous polynomial over $k$ of the same degree
as $f(x,y)$ such that the subset of $\PP^2_k$ given by
$g(x,y,z)=f(x,y)=0$ consists of $d^2$ distinct $\bar k$-points.
Consider the projective surface
$Y \subseteq \PP^2_k \times \PP^1_k$ given by the equation
$$\lambda f(x,y) + \mu g(x,y,z) = 0,$$
where $(\lambda:\mu)$ are homogeneous coordinates on $\PP^1_k$.
One immediately checks that $Y$ is smooth.
The fibre of the projection $Y\to \PP^1_k$ over $\infty=(1:0)$ is $C^f$.

Let $E$ be a smooth, geometrically integral, projective curve over $k$
containing exactly one $k$-rational point $M$. By \cite{P2}
such curves exist over any global field $k$.
Choose a dominant morphism $\vphi: E \to \PP^1_k$
such that $\vphi(M)=\infty$, and $\vphi$ is not ramified over all the
points of $\PP^1_k$ where $Y$ has a singular fibre (including $\infty$).
Define
\begin{equation}
X = E \times_{\PP^1_k} Y. \label{e3}
\end{equation}

\begin{prop} \label{aa}
The surface $X$ is smooth, projective and geometrically integral.
The set $X(k)$ has exactly one point, whereas the set
$X(\A_k)^{\et,\Br}$ is infinite. Furthermore, an infinite subset
of $X(\A_k)^{\et,\Br}$ is contained in the Zariski open
set $X\setminus X(k)$.
\end{prop}
\begin{proof}
The inverse image of the unique $k$-point of $E$ in $X$ is $C^f$,
hence $X$ has exactly one $k$-point. Since the \'etale
Brauer--Manin obstruction is a functor from the category of 
$k$-varieties to the category of sets, the inclusion 
$\iota:C^f \hrar X$ induces an inclusion
$$ \iota_*: C^f(\A_k)=C^f(\A_k)^{\et,\Br} \hrar X(\A_k)^{\et,\Br}. $$
We conclude that $X(\A_k)^{\et,\Br}$ contains infinitely many
points that belong to the Zariski open set $X\setminus X(k)$.
\end{proof}

\begin{rem} Following a suggestion of Ambrus P\'al 
made in response to the first version of this paper we now sketch
a more elementary construction of a counterexample to weak approximation
not explained by the \'etale Brauer--Manin obstruction. Consider 
any irreducible binary quadratic form $f(x,y)$ over $k$. Then our
method produces a smooth, projective and geometrically integral surface
$X$ fibred into conics over $E$. The fibre over the unique $k$-point
of $E$ is the irreducible singular conic $C^f$, hence the set $X(k)$ 
consists of the singular point of $C^f$. Let $D$ be the discriminant of $f(x,y)$.
If $v$ is a finite place of $k$ that splits in $k(\sqrt{D})$,
then $C^f(k_v)$ is the union of two projective lines over $k_v$.
Thus $C^f(\A_k)=C^f(\A_k)^{\et,\Br}$ is infinite, and hence so
is $X(\A_k)^{\et,\Br}$. This construction gives a surface of simpler
geometric structure than the surface in 
Proposition \ref{aa}, but it does not possess the stronger
property of the previous example: here no element of
$X(\A_k)^{\et,\Br}$ is contained in the Zariski 
open set $X\setminus X(k)$.
For this argument we assume that the Jacobian of $E$
has rank $0$ and a finite Shafarevich--Tate
group, e.g. $E$ is an elliptic curve over $\QQ$ of analytic rank $0$.
Let $(M_v)$ be an adelic
point in $X(\A_k)^{\et,\Br}$. Its image $(N_v)$ in $E$
is contained in $E(\A_k)^{\Br}$, but this set is just the connected
component of $0=E(k)$ in $E(\A_k)$, see \cite[Prop. 6.2.4]{Sk}. 
Thus for all finite places $v$ we have $N_v=0$. For any place 
$v$ that does not split in $k(\sqrt{D})$, this
implies $M_v\in X(k)$. 
\end{rem}

\section{The Hasse principle} \label{HP}

In this section we construct a smooth projective surface $X$
over an arbitrary number field $k$ such that $X(\A_k)^{\et,\Br}$ 
is infinite and $X(k)$ is empty. 
This means that $X$ does not satisfy the Hasse 
principle and this failure is not explained by the \'etale
Brauer--Manin obstruction. 

Let $f(x,y)$ and $Z^f$ be as in the previous section.
The scheme $Z^f$ is the disjoint union of $\Spec(K_i)$ for $i=1,\ldots,n$,
where $K_i$ is a finite extension of $k$. We assume $d=\deg(f)\geq 5$.
(For $k=\QQ$ one can take $f(x,y)$ of degree $6$ as in (\ref{e1}),
and in general degree $8$ will suffice.)
We choose field extensions $L/k$ and $F/k$ such that
$L\otimes_k K_i$ and $F\otimes_kK_i$ are fields for all
$i=1,\ldots, n$, and, moreover, $[L:k]=d/2-1$, $[F:k]=d/2$
if $d$ is even, and $[L:k]=(d-1)/2$, $[F:k]=(d+1)/2$ if $d$ is odd.
Fix an embedding $$\Spec(L)\coprod\Spec(F)\hookrightarrow\PP^1_k$$ and
let $D$ be the following curve in $\PP^1_k\times \PP^1_k$:
$$D=\big(Z^f\times \PP^1_k\big)\cup \big(\PP^1_k\times \Spec(L)\big)\cup
\big(\PP^1_k\times\Spec(F)\big).$$
This is a bipartite conical curve without $k$-points, see Definition \ref{bi}.
The class of $D$ in $\Pic(\PP^1_k\times \PP^1_k)$ is
$(d,d-1)$ or $(d,d)$ depending on the parity of $d$.  

\begin{prop}
The natural map $\Br(k)\to \Br(D)$ is an isomorphism.
\end{prop}
\begin{proof} 
Since $D(L)\not=\emptyset$, the natural map
$\Br(L)\to \Br(D\times_k L)$ has a retraction, and so is injective.
The composition of the restriction $\Br(k)\to\Br(L)$ and the corestriction
$\Br(L)\to\Br(k)$ is the multiplication by $[L:k]$, hence for any $x$ in
the kernel of the natural map $\Br(k)\to\Br(D)$ we have $[L:k]x=0$.
Similarly, $D(F)\not=\emptyset$ implies that $[F:k]x=0$.
But $[F:k]=[L:k]+1$, therefore the natural map
$\Br(k)\to \Br(D)$ is injective. By Corollary \ref{c4} 
we need to prove that $\Br(k)$ is the kernel of the map
$$\Br(L)\oplus \Br(F)\oplus \bigoplus_{i=1}^n\Br(K_i)
\lrar \bigoplus_{i=1}^n\Br(LK_i) \oplus \bigoplus_{i=1}^n\Br(FK_i). $$ 
Recall that the maps $\Br(L)\to \Br(LK_i)$ are the restriction maps, whereas
$\Br(K_i)\to\Br(LK_i)$ are opposites of
the restriction maps. The same convention applies with $F$ in place of $L$.

Suppose that we have $\alpha_i\in\Br(K_i)$, $i=1,\ldots,n$,
$\beta\in\Br(L)$ and $\gamma\in\Br(F)$ such that
$(\alpha_i,\beta,\gamma)$ goes to zero.
Let $v$ be a place of $k$, and let $k_v$ be a completion
of $k$ at $v$. Since $Z^f(k_v)\not=\emptyset$,
there is an index $i$ such that $v$ splits in $K_i$.
Let $w$ be a degree 1 place of $K_i$ over $v$, so that the
natural map $k_v\to K_{i,w}$ is an isomorphism. Let
$a_v\in \Br(k_v)=\Br(K_{i,w})$ be
the image of $\alpha_i$ under the restriction map $\Br(K_i)\to \Br(K_{i,w})$.
This defines $a_v\in\Br(k_v)$ for any place $v$ of $k$, moreover,
we have $a_v=0$ for almost all $v$ since each $\alpha_i$ is
unramified away from a finite set of places 
(and the Brauer group of the ring of integers of $k_\nu$ is trivial).
We have a commutative diagram of restriction maps
$$\begin{array}{ccccc}
\Br(L)&\lrar&\Br(L\otimes_k K_i)&\longleftarrow&\Br(K_i)\\
\downarrow&&\downarrow&&\downarrow\\
\Br(L\otimes_k k_v)&\tilde\lrar&\Br(L\otimes_k K_{i,w})
&\longleftarrow&\Br(K_{i,w})
\end{array}$$
Here for a family of fields $\{F_i\}$ we write $\Br(\oplus F_i)=\oplus \Br(F_i)$.
Since $(\alpha_i,\beta,\gamma)$ goes to zero, the restrictions of 
$\alpha_i$ and $\beta$ to
$\Br(L\otimes_k K_i)$ coincide. Hence the image $\beta_v$ of
$\beta$ in $\Br(L\otimes_k k_v)$ is the image of $a_v\in\Br(k_v)$.
From the global reciprocity law applied to $\beta\in \Br(L)$ we
deduce $[L:k]\sum_v\inv_v(a_v)=0$. The same argument with $\gamma$
instead of $\beta$ gives $[F:k]\sum_v\inv_v(a_v)=0$.
Since $[L:k]$ and $[F:k]$ are coprime we obtain $\sum_v\inv_v(a_v)=0$.
By global class field theory there exists $\alpha\in\Br(k)$
such that $a_v$ is the image of $\alpha$ in $\Br(k_v)$.
Since the map $\Br(L)\to\oplus_v\Br(L\otimes_k k_v)$ is injective
it follows that $\alpha$ goes to $\beta$ under the restriction map
$\Br(k)\to\Br(L)$, and similarly for $\gamma$.
Modifying $\alpha_i$, $\beta$ and $\gamma$
by the image of $\alpha$ we can now assume that $\beta=0$ and $\gamma=0$.
Since $\alpha_i$ goes to zero in $\Br(LK_i)$,
a standard restriction-corestriction argument gives
$[L:k]\alpha=0$. Similarly, $\alpha_i$ goes to zero in $\Br(FK_i)$,
and hence $[F:k]\alpha=0$. Therefore, $\alpha=0$.
\end{proof}

\begin{cor} \label{c1}
We have $D(\A_k)^{\Br}=D(\A_k)$.
\end{cor}

We now construct a conical curve $C\subset\PP^1_k\times\PP^1_k$
with one singular point such that $C$ and $D$ are linearly equivalent.
Let $P=(P_1,P_2)$ be a $k$-point in $\PP^1_k\times\PP^1_k$.
In the tangent plane to $\PP^1_k\times\PP^1_k$ at $P$
choose a line $\ell$ through $P$ such that $\ell$ is not one of
the two tangent directions. Assume first that $d$ is odd, so that
$\OO(D)=\OO(d,d)$.
For $i=1,\ldots,d$ let $C_i\subset\PP^1_k\times\PP^1_k$
be pairwise different geometrically irreducible curves through $P$
tangent to $\ell$ such that $\OO(C_i)=\OO(1,1)$. (If one
embeds $\PP^1_k\times\PP^1_k$ as a quadric $Q\subset\PP^3_k$,
then $C_i$ are smooth conic sections of $Q$ by pairwise different
hyperplanes passing through $\ell$.)
Define the curve $C$ as the union of the conics
$C_i$, for $i=1,\ldots,d$. Since $(C_i^2)=2$ and all the curves
$C_i$ are tangent to each other, we have $C_i\cap C_j=P$ if $i\not=j$.
Thus $P$ is the unique singular point of $C$. 
If $d$ is even, we define $C$
as the union of $C_i$, for $i=1,\ldots,d-1$, and $L=P_1\times \PP^1_k$.
We have $L\cap C_i= P$ for $i=1,\ldots,d-1$, so $P$ is a unique
singular point of $C$.
Therefore, for any $d$ the curve $C$ is conical and $X(C)$ is a tree,
and $C$ and $D$ have the same
class in $\Pic(\PP^1_k\times\PP^1_k)$, and so are linearly equivalent.

By choosing $P$ outside $D$ we can arrange that $C$ does not
meet $D_{\mathrm{sing}}$. Then each $\bar k$-point $s$
of $C\cap D$ belongs to exactly one geometric irreducible
component of each $C$ and $D$, and these components meet
transversally at $s$.

Let $r(x,y;u,v)$ and $s(x,y;u,v)$ be the bi-homogeneous polynomials
of bi-degree $(d,d)$ if $d$ is odd, and $(d,d-1)$ if $d$ is even,
such that their zero sets are the curves $D$
and $C$, respectively. Let $Y\subset (\PP^1_k)^3$ be the surface
given by
$$\lambda r(x,y;u,v)+ \mu s(x,y;u,v)=0,$$
where $(\lambda:\mu)$ are homogeneous coordinates on the third copy
of $\PP^1_k$. It is easy to check that $Y$ is smooth.
The projection to the third factor $(\PP^1_k)^3\to\PP^1_k$
defines a surjective morphism $Y\to\PP^1_k$ with fibres
$Y_0=C$ and $Y_\infty=D$. The generic fibre of $Y\to\PP^1_k$
is geometrically integral.

As in the previous section, we pick a smooth, geometrically integral, 
projective curve $E$ with a unique $k$-point $M$, and a dominant morphism
$\vphi: E \to \PP^1_k$
such that $\vphi(M)=\infty$, and $\vphi$ is not ramified over all
points of $\PP^1_k$ where $Y$ has a singular fibre (including $0$ and $\infty$).
We define $X$ by (\ref{e3}); this is smooth, 
geometrically integral and projective surface.
Let $p: X \to E$ be the natural projection. Since $E(k)=\{M\}$ and
$D=p^{-1}(M)$ has no $k$-points, we see that $X(k)=\emptyset$.

\medskip

To study the \'etale Brauer--Manin set of $X$
we need to understand $X$-torsors 
of an arbitrary finite $k$-group scheme $G$. In the following 
general proposition the word `torsor' means `torsor 
with structure group $G$'.

\begin{prop} \label{p2}
Let $k$ be a field of characteristic zero, and let $X$ and $B$
be varieties over $k$. Let $p:X\to B$ be a proper morphism 
with geometrically connected fibres. Assume that $p$ has
a simply connected geometric fibre. Then for any torsor
$f:X' \to X$ there exists a torsor $B'\to B$ 
such that torsors $X'\to X$ and $X\times_B B'\to X$ are isomorphic.
\end{prop}
\begin{proof}
Let $\delta=|G(\bar k)|$. Let 
$X'\stackrel{g}{\to} B'\stackrel{h}{\to} B$ be the Stein factorisation of
the composed morphism $X'\stackrel{f}{\to} X \stackrel{p}{\to} B$
(EGA III.4.3.1). Thus we have a commutative diagram
$$ \xymatrix{
X' \ar^{f}[r]\ar_{g}[d] & X \ar^{p}[d] \\
B' \ar^{h}[r] & B \\
}$$
where $g$ is proper with geometrically connected fibres, 
and $h$ is a finite morphism. (The variety $B'$ can be defined
as the relative spectrum of $(pf)_*\O_{X'}$.)
The Stein factorisation is a functor from the category of 
proper schemes over $B$ to the category of finite schemes over $B$.
Thus we obtain an induced action of $G$ on $B'$ such that $g$ is $G$-equivariant.

Let $\delta'$ be the degree of $h$.
We claim that $\delta'\geq\delta$. Indeed, 
let $\bar x$ be a geometric point of $B$ such that 
$p^{-1}(\bar x)$ is simply connected. Thus
$f^{-1}(p^{-1}(\bar x))$ is isomorphic to a disjoint union of 
$\delta$ copies of $p^{-1}(\bar x)$. 
Hence $h^{-1}(\bar x)$ has cardinality 
$\delta$, and so $\delta'\geq\delta$.

The projection $\pi_X: X\times_B B'\to X$ is a finite morphism of 
degree $\delta'$. 
The composition of the natural map $X'\to X\times_B B'$ with $\pi_X$ 
is $f:X'\to X$, and this is a finite \'etale morphism of degree $\delta$.
Since $\delta'\geq\delta$ it follows that 
$X'\to X\times_B B'$ is finite and \'etale of degree $1$, and hence 
is an isomorphism. We also see that $\delta'=\delta$ and that $\pi_X$ 
is an \'etale morphism of degree $\delta$.
It follows that $h$ is also \'etale of degree $\delta$. 
Now the action of $G$ equips 
$B'$ with the structure of a $B$-torsor.
\end{proof}

Finally, we can prove the main result of this section.

\begin{thm}
The set $X(k)$ is empty, whereas the set $X(\A_k)^{\et,\Br}$ contains 
$D(\A_k)$ and so is infinite.
\end{thm}
\begin{proof}
Since $\ovl C$ is a connected and simply connected variety over
$\bar k$ which is a geometric fibre of $X\to E$, we can use
Proposition \ref{p2}. Thus any $X$-torsor of a finite
$k$-group scheme has the form
$\T_X=\T\times_E X$ for some torsor $\T\to E$.
After twisting we can assume that the fibre of $\T$ over
the unique $k$-point of $E$ has a $k$-point.
Thus the restriction of the torsor $\T_X\to X$
to the curve $D\subset X$ has a section $\sigma$.
Therefore every adelic point on $D\subset X$ is the image of
an adelic point on $\sigma(D)\subset \T_X$.
By Corollary \ref{c1} and the functoriality of the Brauer--Manin
set we have
$$\sigma(D)(\A_k)=\sigma(D)(\A_k)^{\Br}\subset \T_X(\A_k)^{\Br}.$$
Thus, by the definition of the \'etale Brauer--Manin set,
$X(\A_k)^{\et,\Br}$ contains the infinite set $D(\A_k)$, and hence is infinite.
\end{proof}


\begin{thebibliography}{cc}

\bibitem{BSk} C.L. Basile and A.N. Skorobogatov.
On the Hasse principle for bielliptic surfaces. 
{\em Number theory and algebraic geometry}, 31--40,
London Math. Soc. Lecture Note Ser. {\bf 303}, 
Cambridge University Press, 2003.

\bibitem{BLR} S. Bosch, W. L\"utkebohmert and M. Raynaud. {\em N\'eron models.}
Ergebnisse der Mathematik und Ihrer Grenzgebiete, Springer-Verlag, 1990.

\bibitem{D} C. Demarche. Obstruction de descente et obstruction de 
Brauer--Manin \'etale. {\em J. Algebra Number Theory} {\bf 3} (2009) 237--254.

\bibitem{H} D. Harari. Weak approximation and non-abelian fundamental
groups. {\em Ann. Sci. \'Ecole Norm. Sup.} {\bf 33} (2000) 467--484.

\bibitem{H2} D. Harari. Groupes alg\'ebriques et points rationnels. 
{\em Math. Ann.} {\bf 322} (2002) 811--826.

\bibitem{Y} Y. Harpaz. The section conjecture for graphs and conical curves.
arXiv:1304.7213

\bibitem{GB} A. Grothendieck. Le groupe de Brauer. 
{\em Dix expos\'es sur la cohomologie
des sch\'emas}, 46--188, North-Holland, 1968.

\bibitem{FE}
D. Ferrand. Conducteur, descente et pincement. {\em Bull. Soc. math. France} {\bf 131} (2003) 553--585.

\bibitem{Liu} Q. Liu. {\em Algebraic geometry and arithmetic curves.}
Oxford University Press, 2010.

\bibitem{Manin} Yu.I. Manin. Le groupe de Brauer--Grothendieck en
g\'eom\'etrie diophantienne. {\em Actes du Congr\`es International des
Math\'ematiciens} (Nice, 1970), Tome 1, 401--411. Gauthier-Villars, 1971.

\bibitem{EC} J.S. Milne. {\em \'Etale cohomology.} Princeton University Press, 1980.

\bibitem{P1} B. Poonen. Insufficiency of the Brauer--Manin obstruction applied to \'etale covers.
{\em Ann. Math.} {\bf 171} (2010) 2157--2169.

\bibitem{P2} B. Poonen. Curves over every global field violating the local-global principle.
{\em Zap. Nauchn. Sem. POMI} {\bf 377} (2010),
Issledovaniya po Teorii Chisel. {\bf 10} 141--147, 243--244. English translation:
J. Math. Sci. (N.Y.) {\bf 171} (2010) 782--785.

\bibitem{SW} P. Sarnak and L. Wang. Some hypersurfaces in $\PP^4$ and the Hasse-principle.
{\em C. R. Acad. Sci. Paris S\'er. I Math.} {\bf 321} (1995) 319--322.

\bibitem{CL} J.-P. Serre. {\em Corps locaux.} Hermann, Paris, 1962.

\bibitem{Sk1} A.N. Skorobogatov. Beyond the Manin obstruction. 
{\em Invent. Math.} {\bf 135} (1999) 399--424.

\bibitem{Sk} A. Skorobogatov. {\em Torsors and rational points}. Cambridge
University Press, 2001.

\bibitem{Sk2} A.N. Skorobogatov. Descent obstruction is equivalent to 
\'etale Brauer--Manin obstruction. {\em Math. Ann.} {\bf 344} (2009) 501--510.

\bibitem{Stoll} M. Stoll. Finite descent obstructions and rational 
points on curves. {\em J. Algebra Number Theory} {\bf 1}
(2007) 349--391.

\end{thebibliography}
\end{document}